\documentclass[11pt,a4paper]{article}

\usepackage{inputenc}
\usepackage{amsmath}
\usepackage{bm}
\usepackage{bbold}
\usepackage{amsthm}
\usepackage{enumerate}

\usepackage{hyperref}

\title{A Solution of a Tropical Linear Vector Equation\thanks{Advances in Computer Science: Proc. 6th WSEAS European Computing Conf. (ECC '12), WSEAS Press, 2012, pp.~244--249. (Recent Advances in Computer Engineering Series, Vol.~5)}
} 

\author{Nikolai Krivulin\thanks{Faculty of Mathematics and Mechanics, St.~Petersburg State University, 28 Universitetsky Ave., St.~Petersburg, 198504, Russia, 
nkk@math.spbu.ru.}
}

\date{}

\newtheorem{theorem}{Theorem}
\newtheorem{lemma}[theorem]{Lemma}

\newtheorem{proposition}{Proposition}

\begin{document}

\maketitle

\begin{abstract}
A linear vector equation is considered defined in terms of idempotent mathematics. To solve the equation, we apply an approach that is based on the analysis of distances between vectors in idempotent vector spaces and reduces the solution of the equation to that of a tropical optimization problem. Based on the approach, existence and uniqueness conditions are established for the solution, and a general solution to the equation is given.  
\\

\end{abstract}

\section{Introduction}

Many applications of tropical mathematics \cite{Vorobjev1967Extremal,Cuninghame-green1979Minimax,Baccelli1993Synchronization,Litvinov1998Correspondence,Golan2003Semirings,Heidergott2006Max-plus,Butkovic2010Maxlinear} involve the solution of  linear equations defined on finite-dimensional semimodules over idempotent semifields (idempotent vector spases). One of the equation that often arise has the form $A\bm{x}=\bm{d}$, where $A$ and $\bm{d}$ are given matrix and vector, $\bm{x}$ is an unknown vector, and multiplication is thought of in terms of idempotent algebra. Since the equation can be considered as representing linear dependence between vectors, the development of efficient solution methods for the equation is of both practical and theoretical interest.

There are several existing solution methods, including those in \cite{Vorobjev1967Extremal,Cuninghame-green1979Minimax,Baccelli1993Synchronization}. In this paper another solution approach is described based on the analysis of distances between vectors in idempotent vector spaces. The results presented are based on implementation and further refinement of solutions published in the papers \cite{Krivulin2005Onsolution,Krivulin2009Onsolution,Krivulin2009Methods} and not available in English.

We start with a brief overview of preliminary algebraic definitions and results. Furthermore, the problem of solving the equation under study reduces to an optimization problem of finding the minimal distance from a vector to a linear span of vectors. We derive a comprehensive solutions to the optimization problem under quite general conditions. The obtained results are applied to give existence and uniqueness condition as well as to offer a general solution of the equation.

\section{Preliminaries}

In this section, we present algebraic definitions, notations, and results based on \cite{Krivulin2009Onsolution,Krivulin2009Methods} to provide a background for subsequent analysis and solutions. Additional details and further results are found in 
\cite{Vorobjev1967Extremal,Cuninghame-green1979Minimax,Baccelli1993Synchronization,Litvinov1998Correspondence,Golan2003Semirings,Heidergott2006Max-plus,Butkovic2010Maxlinear}.

We consider a set $\mathbb{X}$ endowed with addition $\oplus$ and multiplication $\otimes$ and equipped with the zero $\mathbb{0}$ and the identity $\mathbb{1}$. The system $\langle\mathbb{X},\mathbb{0},\mathbb{1},\oplus,\otimes\rangle$ is assumed to be a linearly ordered radicable commutative semiring with idempotent addition and invertible multiplication (idempotent semifield).

Idempotency of addition implies that $x\oplus x=x$ for all $x\in\mathbb{X}$. For any $x\in\mathbb{X}_{+}$, where $\mathbb{X}_{+}=\mathbb{X}\setminus\{\mathbb{0}\}$, there exists an inverse $x^{-1}$ such that $x^{-1}\otimes x=\mathbb{1}$. Furthermore, the power $x^{q}$ is defined for any $x\in\mathbb{X}_{+}$ and a rational $q$. Specifically, for any integer $p\geq0$, we have $x^{0}=\mathbb{1}$, $x^{p}=x^{p-1}x$, $x^{-p}=(x^{-1})^{p}$.

In what follows, we drop the multiplication sign $\otimes$ and use the power notation only in the above sense.

The linear order defined on $\mathbb{X}$ is assumed to be consistent with a partial order that is induced by idempotent addition and involves that $x\leq y$ if and only if $x\oplus y=y$. Below, the relation symbols and the operator $\min$ are thought in terms of this linear order.

As an example of the semifields under consideration, one can consider a semifield of real numbers $\mathbb{R}_{\max,+}=\langle\mathbb{R}\cup\{-\infty\},-\infty,0,\max,+\rangle$.

\subsection{Idempotent Vector Space}

Consider the Cartesian power $\mathbb{X}^{m}$ with column vectors as its elements. A vector with all components equal to $\mathbb{0}$ is called the zero vector. A vector is regular if it has no no zero components. 

For any two vectors $\bm{a}=(a_{i})$ and $\bm{b}=(b_{i})$ in $\mathbb{X}^{m}$, and a scalar $x\in\mathbb{X}$, addition and scalar multiplication are defined componentwise as
$$
\{\bm{a}\oplus\bm{b}\}_{i}
=
a_{i}\oplus b_{i},
\qquad
\{x\bm{a}\}_{i}
=
xa_{i}.
$$

Endowed with these operations, the set $\mathbb{X}^{m}$ forms a semimodule over the idempotent semifield $\mathbb{X}$, which is referred to as the idempotent vector space.

Fig.~\ref{F-VASM} illustrates the operations in the space $\mathbb{R}_{\max,+}^{2}$ with the Cartesian coordinates on the plane.
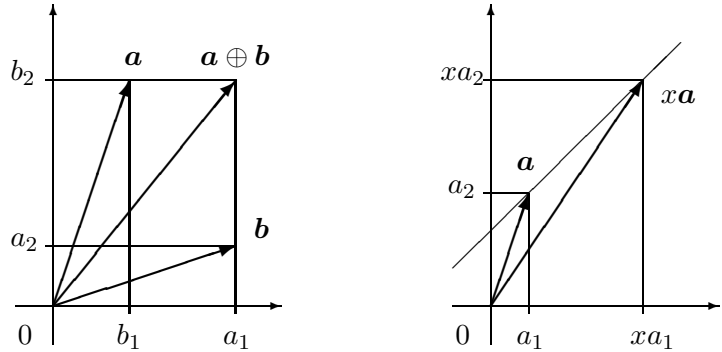
\begin{figure}[ht]
\setlength{\unitlength}{1mm}
\begin{center}
\begin{picture}(35,45)

\put(0,5){\vector(1,0){35}}
\put(5,0){\vector(0,1){45}}

\put(5,5){\thicklines\vector(1,3){10}}
\put(15,35){\line(0,-1){31}}

\put(5,5){\thicklines\vector(3,1){24}}
\put(29,13){\line(-1,0){25}}

\put(5,5){\thicklines\line(4,5){24}}
\put(26,31.75){\thicklines\vector(1,1){3}}

\put(29,35){\line(-1,0){25}}

\put(29,35){\line(0,-1){31}}

\put(0,0){$0$}

\put(13,0){$b_{1}$}
\put(27,0){$a_{1}$}

\put(-1,13){$a_{2}$}
\put(-1,35){$b_{2}$}

\put(14,37){$\bm{a}$}

\put(31,14){$\bm{b}$}

\put(24,37){$\bm{a}\oplus\bm{b}$}

\end{picture}
\hspace{20\unitlength}
\begin{picture}(35,45)

\put(0,5){\vector(1,0){35}}
\put(5,0){\vector(0,1){45}}

\put(5,5){\thicklines\vector(1,3){5}}
\put(10,20){\line(0,-1){16}}
\put(10,20){\line(-1,0){6}}

\put(5,5){\thicklines\vector(2,3){20}}
\put(25,35){\line(0,-1){31}}
\put(25,35){\line(-1,0){21}}

\put(0,10){\line(1,1){30}}

\put(0,0){$0$}
\put(8,23){$\bm{a}$}
\put(27,32){$x\bm{a}$}

\put(-1,20){$a_{2}$}
\put(-2,35){$xa_{2}$}

\put(8,0){$a_{1}$}
\put(23,0){$xa_{1}$}

\end{picture}
\end{center}
\caption{Vector addition (left) and multiplication by scalars (right) in $\mathbb{R}_{\max,+}^{2}$.}\label{F-VASM}
\end{figure}

Multiplication of a matrix $A=(a_{ij})\in\mathbb{X}^{m\times n}$ by a vector $\bm{x}=(x_{i})\in\mathbb{X}^{n}$ is routinely defined to result in a vector with components
$$
\{A\bm{x}\}_{i}
=
a_{i1}x_{1}\oplus\cdots\oplus a_{in}x_{n}.
$$

All above operations are componentwise isotone in each argument.

A matrix is regular if it has no zero rows.

For any nonzero column vector $\bm{x}=(x_{i})\in\mathbb{X}^{n}$, we define a row vector $\bm{x}^{-}=(x_{i}^{-})$, where $x_{i}^{-}=x_{i}^{-1}$ if $x_{i}\ne\mathbb{0}$, and $x_{i}^{-}=\mathbb{0}$ otherwise.

For any two regular vectors $\bm{x}$ and $\bm{y}$, the componentwise inequality $\bm{x}\leq\bm{y}$ implies $\bm{x}^{-}\geq\bm{y}^{-}$.

If $\bm{x}$ is a regular vector, then $\bm{x}^{-}\bm{x}=\mathbb{1}$ and $\bm{x}\bm{x}^{-}\geq I$, where $I$ is an identity matrix with the elements equal to  $\mathbb{1}$ on the diagonal, and $\mathbb{0}$ elsewhere.

\subsection{Linear Dependence}

Consider a system of vectors $\bm{a}_{1},\ldots,\bm{a}_{n}\in\mathbb{X}^{m}$. As usual, a vector $\bm{b}\in\mathbb{X}^{m}$ is linearly dependent on the system if it admits representation  as a linear combination $\bm{b}=x_{1}\bm{a}_{1}\oplus\cdots\oplus x_{n}\bm{a}_{n}$, where $x_{1},\ldots x_{n}\in\mathbb{X}$.

The set of all linear combinations of $\bm{a}_{1},\ldots,\bm{a}_{n}$ form a linear span denoted by $\mathop\mathrm{span}\{\bm{a}_{1},\ldots,\bm{a}_{n}\}$. An example of a linear span in $\mathbb{R}_{\max,+}^{2}$ is given in Fig.~\ref{F-LC}.
\begin{figure}[ht]
\setlength{\unitlength}{1mm}
\begin{center}

\begin{picture}(50,40)

\put(0,5){\vector(1,0){50}}
\put(5,0){\vector(0,1){40}}

\put(5,5){\thicklines\vector(1,3){4.25}}
\put(5,5){\thicklines\vector(2,3){17}}

\put(5,5){\thicklines\vector(4,1){23}}
\put(5,5){\thicklines\vector(2,1){34}}

\put(1.5,10){\thicklines\line(1,1){26}}
\multiput(2.5,11)(1,1){25}{\line(1,0){1}}

\put(17,0){\thicklines\line(1,1){26}}
\multiput(18,1)(1,1){25}{\line(-1,0){1}}

\put(39,30.5){\line(-1,0){35}}
\put(39,30.5){\line(0,-1){26.5}}

\put(10,18){\line(-1,0){6}}
\put(28,11){\line(0,-1){7}}

\put(5,5){\thicklines\vector(4,3){34}}

\put(0,0){$0$}

\put(7,21){$\bm{a}_{2}$}
\put(0,24){$x_{2}$}
\put(15,34){$x_{2}\bm{a}_{2}$}

\put(30,9){$\bm{a}_{1}$}
\put(32,0){$x_{1}$}
\put(42,20){$x_{1}\bm{a}_{1}$}

\put(31,34){$x_{1}\bm{a}_{1}\oplus x_{2}\bm{a}_{2}$}

\end{picture}
\end{center}
\caption{A linear span of vectors $\bm{a}_{1},\bm{a}_{2}$ in $\mathbb{R}_{\max,+}^{2}$.}\label{F-LC}
\end{figure}
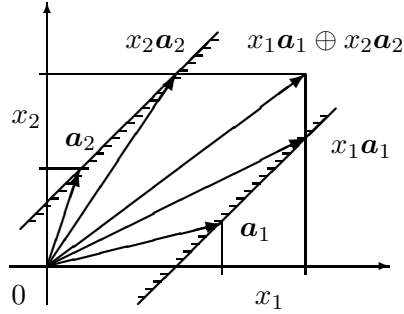

A system of vectors $\bm{a}_{1},\ldots,\bm{a}_{n}$ is linearly dependent if at least one its vector is linearly dependent on others, and it is linear independent otherwise.

A system of nonzero vectors $\bm{a}_{1},\ldots,\bm{a}_{n}$ is a minimal generating system for a vector $\bm{b}$, if this vector is linearly dependent on the system and independent of any of its subsystems.

Let us verify that if vectors $\bm{a}_{1},\ldots,\bm{a}_{n}$ are a minimal generating system for a vector $\bm{b}$, then representation of $\bm{b}$ as a linear combination of $\bm{a}_{1},\ldots,\bm{a}_{n}$ is unique. Suppose there are two linear combinations
$$
\bm{b}=x_{1}\bm{a}_{1}\oplus\cdots\oplus x_{n}\bm{a}_{n}=x_{1}^{\prime}\bm{a}_{1}\oplus\cdots\oplus x_{n}^{\prime}\bm{a}_{n},
$$
where $x_{i}^{\prime}\ne x_{i}$ for some index $i=1,\ldots,n$.

Assuming for definiteness that $x_{i}^{\prime}<x_{i}$, we have $\bm{b}\geq x_{i}\bm{a}_{i}>x_{i}^{\prime}\bm{a}_{i}$. Therefore, the term $x_{i}^{\prime}\bm{a}_{i}$ does not affect $\bm{b}$ and so may be omitted, which contradicts with the minimality of the system $\bm{a}_{1},\ldots,\bm{a}_{n}$.

\subsection{Distance Function}

For any vector $\bm{a}\in\mathbb{X}^{m}$, we introduce the support as the index set $\mathop\mathrm{supp}(\bm{a})=\{i|a_{i}\ne\mathbb{0}\}$.

The distance between nonzero vectors $\bm{a},\bm{b}\in\mathbb{X}^{m}$ with $\mathop\mathrm{supp}(\bm{a})=\mathop\mathrm{supp}(\bm{b})$ is defined by a function
\begin{equation}
\rho(\bm{a},\bm{b})
=
\bigoplus_{i\in\mathop\mathrm{supp}(\bm{a})}\left(b_{i}^{-1} a_{i}\oplus a_{i}^{-1} b_{i}\right).
\label{E-rhoab}
\end{equation}

We put $ \rho(\bm{a},\bm{b})=\infty$ if $\mathop\mathrm{supp}(\bm{a})\ne\mathop\mathrm{supp}(\bm{b})$, and $\rho(\bm{a},\bm{b})=\mathbb{1}$ if $\bm{a}=\bm{b}=\mathbb{0}$.

Note that in $\mathbb{R}_{\max,+}^{m}$, the function $\rho$ coincides for all vectors $\bm{a},\bm{b}\in\mathbb{R}^{m}$ with the Chebyshev metric
$$
\rho_{\infty}(\bm{a},\bm{b})
=
\max_{1\leq i\leq m}|b_{i}-a_{i}|.
$$

\section{Evaluation of Distances}

Let $\bm{a}_{1},\ldots,\bm{a}_{n}\in\mathbb{X}^{m}$ be given vectors. Denote by $A=(\bm{a}_{1},\ldots,\bm{a}_{n})$ a matrix composed of the vectors, and by $\mathcal{A}=\mathop\mathrm{span}\{\bm{a}_{1},\ldots,\bm{a}_{n}\}$ their linear span.

Take a vector $\bm{d}\in\mathbb{X}^{m}$ and consider the problem of computing the distance from $\bm{d}$ to $\mathcal{A}$ defined as
$$
\rho(\mathcal{A},\bm{d})
=
\min_{\bm{a}\in\mathcal{A}}\rho(\bm{a},\bm{d}).
$$

Taking into account that every vector $\bm{a}\in\mathcal{A}$ can be represented as $\bm{a}=A\bm{x}$ for some vector $\bm{x}\in\mathbb{X}^{n}$, we arrive at the problem of calculating
\begin{equation}
\rho(\mathcal{A},\bm{d})
=
\min_{\bm{x}\in\mathbb{X}^{n}}\rho(A\bm{x},\bm{d}).
\label{E-rhoAd}
\end{equation}

Suppose $\bm{d}=\mathbb{0}$. Considering that $\mathcal{A}$ always contains the zero vector, we obviously get $\rho(\mathcal{A},\bm{d})=\mathbb{1}$.

Let some of the vectors $\bm{a}_{1},\ldots,\bm{a}_{n}$ be zero. Since zero vectors do not affect the linear span $\mathcal{A}$, they can be removed with no change of distances. When all vectors are zero and thus $\mathcal{A}=\{\mathbb{0}\}$, we have $\rho(\mathcal{A},\bm{d})=\mathbb{0}$ if $\bm{d}=\mathbb{0}$, and $\rho(\mathcal{A},\bm{d})=\infty$ otherwise.

From here on we assume that $\bm{d}\ne\mathbb{0}$ and $\bm{a}_{i}\ne\mathbb{0}$ for all $i=1,\ldots,n$.

Suppose the vector $\bm{d}=(d_{i})$ may have zero components and so be irregular. For the matrix $A=(a_{ij})$, we introduce a matrix $\widehat{A}=(\widehat{a}_{ij})$ as follows. We define two sets of indices $I=\{i|d_{i}=\mathbb{0}\}$ and $J=\{j|a_{ij}>\mathbb{0}, i\in I\}$, and then deteremine the entries in $\widehat{A}$ according to the conditions
$$
\widehat{a}_{ij}
=
\begin{cases}
\mathbb{0},	& \text{if $i\notin I$ and $j\in J$}, \\
a_{ij},			& \text{otherwise}.
\end{cases}
$$

The matrix $A$ may differ from $\widehat{A}$ only in those columns that have nonzero intersections with the rows corresponding to zero components in $\bm{d}$. In the matrix $\widehat{A}$, these columns have all entries that are not located at the intersections set to zero. The matrix $\widehat{A}$ and the vector $\bm{d}$ are said to be consistent with each other. 

Note that when $\bm{d}$ is regular, we have $\widehat{A}=A$.

\begin{proposition}
For all $\bm{x}=(x_{i})\in\mathbb{X}^{n}$ it holds that
$$
\rho(A\bm{x},\bm{d})
=
\rho(\widehat{A}\bm{x},\bm{d}).
$$
\end{proposition}
\begin{proof}
With a regular $\bm{d}$ the statement becomes trivial and so assume $\bm{d}\ne\mathbb{0}$ to have zero components.

Suppose that $\rho(A\bm{x},\bm{d})<\infty$, which occurs only under the condition $\mathop\mathrm{supp}(A\bm{x})=\mathop\mathrm{supp}(\bm{d})$. The fulfillment of the condition is equivalent to equalities $a_{i1}x_{1}\oplus\cdots\oplus a_{in}x_{n}=\mathbb{0}$ that must be true whenever $d_{i}=\mathbb{0}$. To provide the equalities, we must put $x_{j}=\mathbb{0}$ for all indices $j$ such that $a_{ij}\ne\mathbb{0}$ for at least one index $i$ with $d_{i}=\mathbb{0}$. In this case, replacing $A$ with $\widehat{A}$ leaves the value of $\rho(A\bm{x},\bm{d})<\infty$ unchanged.

Since the condition $\mathop\mathrm{supp}(A\bm{x})\ne\mathop\mathrm{supp}(\bm{d})$ implies $\mathop\mathrm{supp}(\widehat{A}\bm{x})\ne\mathop\mathrm{supp}(\bm{d})$ and vice versa, the statement is also true when $\rho(A\bm{x},\bm{d})=\infty$.
\end{proof}

With the above result, we may now concentrate only on the problems when $A$ is consistent with $\bm{d}$.

In order to describe the solution of problem \eqref{E-rhoAd}, we need the following notation. For any consistent matrix $A$ and a vector $\bm{d}$, we define a residual value
$$
\Delta_{A}(\bm{d})
=
\sqrt{(A(\bm{d}^{-}A)^{-})^{-}\bm{d}}
$$
if $A$ is regular, and $\Delta_{A}(\bm{d})=\infty$ otherwise.

In what follows, we drop subscripts and arguments in $\Delta_{A}(\bm{d})$ and write $\Delta$ if no confusion arises.
 
Below we find the solution when the vector $\bm{d}$ is regular and then extend this result to irregular vectors.
 
\subsection{Regular Vector}

Suppose that $\bm{d}$ is a regular vector. First we verify that the minimum of $\rho(A\bm{x},\bm{d})$ over $\mathbb{X}^{n}$ in \eqref{E-rhoAd} can be found by examining only regular vectors $\bm{x}\in\mathbb{X}_{+}^{n}$.

\begin{proposition}
If the vector $\bm{d}$ is regular, then
$$
\rho(\mathcal{A},\bm{d})
=
\min_{\bm{x}\in\mathbb{X}_{+}^{n}}\rho(A\bm{x},\bm{d}).
$$
\end{proposition}
\begin{proof}
Take a vector $\bm{y}=A\bm{x}$ such that $\rho(A\bm{x},\bm{d})$ achieves the minimum value. If $\bm{y}$ is irregular and so has zero components, then $\mathop\mathrm{supp}(\bm{y})\ne\mathop\mathrm{supp}(\bm{d})$, and thus $\rho(A\bm{x},\bm{d})=\infty$ for all $\bm{x}$, including $\bm{x}>\mathbb{0}$.

Suppose $\bm{y}=(y_{1},\ldots,y_{m})^{T}>\mathbb{0}$, and assume a corresponding vector $\bm{x}$ to have a zero component, say $x_{j}=\mathbb{0}$. We define the set $I=\{i|a_{ij}>\mathbb{0}\}\ne\emptyset$ and find the number $\varepsilon=\min\{a_{ij}^{-1}y_{i}|i\in I\}>\mathbb{0}$.

It remains to note that with $x_{j}=\varepsilon$ in place of $x_{j}=\mathbb{0}$, all components of $\bm{y}$ together with the minimum value of $\rho(A\bm{x},\bm{d})$ remain unchanged.
\end{proof}

The next statement reveals the meaning of the residual $\Delta=\Delta_{A}(\bm{d})$ in terms of distances.
\begin{lemma}\label{L-Lb}
If the vector $\bm{d}$ is regular, then it holds that
$$
\rho(\mathcal{A},\bm{d})
=
\min_{\bm{x}\in\mathbb{X}_{+}^{n}}\rho(A\bm{x},\bm{d})
=
\Delta,
$$
where the minimum is attained at $\bm{x}=\Delta(\bm{d}^{-}A)^{-}$.
\end{lemma}
\begin{proof}
Let the matrix $A$ be irregular. Then we have $\mathop\mathrm{supp}(A\bm{x})\ne\mathop\mathrm{supp}(\bm{d})$ and $\rho(\mathcal{A},\bm{d})=\infty$. Since, by definition, $\Delta=\infty$, the statement is true in this case.

Suppose $A$ is regular. Taking into account \eqref{E-rhoab}, we arrive at an optimization problem to find
$$
\min_{\bm{x}\in\mathbb{X}_{+}^{n}}\ (\bm{d}^{-}A\bm{x}\oplus(A\bm{x})^{-}\bm{d}).
$$ 

Take any vector $\bm{y}=A\bm{x}$ such that $\bm{x}>\mathbb{0}$, and define $r=\bm{d}^{-}A\bm{x}\oplus(A\bm{x})^{-}\bm{d}>\mathbb{0}$.

From the definition of $r$, we have two inequalities
$$
r\geq\bm{d}^{-}A\bm{x},
\qquad
r\geq(A\bm{x})^{-}\bm{d}.
$$

Right multiplication of the first inequality by $\bm{x}^{-}$ gives $r\bm{x}^{-}\geq\bm{d}^{-}A\bm{x}\bm{x}^{-}\geq\bm{d}^{-}A$. Then we obtain $\bm{x}\leq r(\bm{d}^{-}A)^{-}$ and $(A\bm{x})^{-}\geq r^{-1}(A(\bm{d}^{-}A)^{-})^{-}$.

Substitution into the second inequality results in $r\geq r^{-1}(A(\bm{d}^{-}A)^{-})^{-}\bm{d}=r^{-1}\Delta^{2}$, and so in $r\geq\Delta$.

It remains to verify that $r=\Delta$ when we take $\bm{x}=\Delta(\bm{d}^{-}A)^{-}$. Indeed, substitution of the $\bm{x}$ gives $r=\Delta\bm{d}^{-}A(\bm{d}^{-}A)^{-}\oplus\Delta^{-1}(A(\bm{d}^{-}A)^{-})^{-}\bm{d}=\Delta$.

Finally note that the above vector $\bm{x}$ corresponds to the vector $\bm{y}=\Delta A(\bm{d}^{-}A)^{-}\in\mathcal{A}$.
\end{proof}

Examples of a linear span $\mathcal{A}=\mathop\mathrm{span}(\bm{a}_{1},\bm{a}_{2})$ and vectors $\bm{d}$ in the space $\mathbb{R}_{\max,+}^{2}$ are given in Fig.~\ref{F-Lb}.
\begin{figure}[ht]
\setlength{\unitlength}{1mm}
\begin{center}
\begin{picture}(35,45)

\put(0,5){\vector(1,0){35}}
\put(5,0){\vector(0,1){45}}

\put(5,5){\thicklines\vector(1,4){4}}

\put(5,5){\thicklines\vector(2,-1){5}}

\put(0,12){\line(1,1){25}}
\put(0,12){\thicklines\line(1,1){25}}
\multiput(1,13)(1,1){24}{\line(1,0){1}}

\put(7.5,0){\line(1,1){25}}
\put(7.5,0){\thicklines\line(1,1){25}}
\multiput(8,0.5)(1,1){25}{\line(-1,0){1}}

\put(5,5){\thicklines\vector(3,4){16}}

\put(13,0){$\bm{a}_{1}$}
\put(7,26){$\bm{a}_{2}$}
\put(22,25){$\bm{d}$}

\put(28,31){$\mathcal{A}$}

\put(9,38){$\Delta=\mathbb{1}$}

\end{picture}
\hspace{20\unitlength}
\begin{picture}(40,45)

\put(0,5){\vector(1,0){40}}
\put(5,0){\vector(0,1){45}}

\put(5,5){\thicklines\vector(1,4){4}}

\put(5,5){\thicklines\vector(2,-1){5}}

\put(0,12){\line(1,1){25}}
\put(0,12){\thicklines\line(1,1){25}}
\multiput(1,13)(1,1){24}{\line(1,0){1}}

\put(7.5,0){\line(1,1){25}}
\multiput(8,0.5)(1,1){25}{\line(-1,0){1}}
\put(7.5,0){\thicklines\line(1,1){25}}

\put(27.5,20){\line(1,-1){7.5}}

\put(5,5){\thicklines\vector(3,2){22.5}}

\put(5,5){\thicklines\vector(4,1){30}}

\multiput(27.5,20)(0,-2.8){6}{\line(0,-1){2}}

\multiput(35,12.5)(0,-3.2){3}{\line(0,-1){2.2}}

\put(13,0){$\bm{a}_{1}$}
\put(7,26){$\bm{a}_{2}$}
\put(35,14){$\bm{d}$}

\put(25,23){$\bm{y}$}

\put(28,31){$\mathcal{A}$}

\put(9,38){$\Delta>\mathbb{1}$}

\put(29.5,0){$\Delta$}

\end{picture}
\end{center}
\caption{A linear span $\mathcal{A}$ and vectors $\bm{d}$ in $\mathbb{R}_{\max,+}^{2}$ when $\Delta=\mathbb{1}$ (left) and $\Delta>\mathbb{1}$ (right).}\label{F-Lb}
\end{figure}
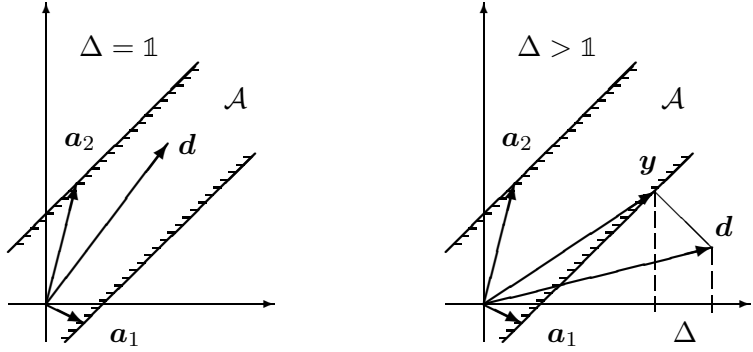

\subsection{Arbitrary Nonzero Vector}

Now we examine the distance between the linear span $\mathcal{A}$ and a vector $\bm{b}\ne\mathbb{0}$. It will suffice to consider only the case when the matrix $A$ is consistent with $\bm{b}$.

\begin{theorem}\label{T-Lb}
Suppose the matrix $A$ is consistent with the vector $\bm{d}\ne\mathbb{0}$. Then it holds that
$$
\rho(\mathcal{A},\bm{d})
=
\min_{\bm{x}\in\mathbb{X}_{+}^{n}}\rho(A\bm{x},\bm{d})
=
\Delta,
$$
where the minimum is attained at $\bm{x}=\Delta(\bm{d}^{-}A)^{-}$.
\end{theorem}
\begin{proof}
Note that for the case when $\bm{d}$ is regular, the proof is given in Lemma~\ref{L-Lb}. Now we suppose that the vector $\bm{d}\ne\mathbb{0}$ has zero components. We define sets of indices $I=\{i|b_{i}=\mathbb{0}\}$ and $J=\{j|a_{ij}>\mathbb{0}, i\in I\}$.

To provide minimum of $\rho(A\bm{x},\bm{d})$, we have to set $x_{j}=\mathbb{0}$ for all $j\in J$. This makes it possible to exclude from consideration all  components of $\bm{d}$ and the rows of $A$ with indices in $I$, as well as all columns of $A$ with indices in $J$. By eliminating these elements, we obtain a new matrix $A^{\prime}$ and a new vector $\bm{d}^{\prime}$.

Denote the linear span of the columns in $A^{\prime}$ by $\mathcal{A}^{\prime}$. Considering that the vector $\bm{d}^{\prime}$ has no zero components, we apply Lemma~\ref{L-Lb} to get
$$
\rho(\mathcal{A},\bm{d})
=
\rho(\mathcal{A}^{\prime},\bm{d}^{\prime})
=
\Delta_{A^{\prime}}(\bm{d}^{\prime})
=
\Delta^{\prime}.
$$

Furthermore, we note that the minimum $\rho(A^{\prime}\bm{x}^{\prime},\bm{d}^{\prime})$ is attained if  $\bm{x}^{\prime}=\Delta^{\prime}(\bm{d}^{\prime-}A^{\prime})^{-}$, where $\bm{x}^{\prime}$ is a vector of order less than $n$.

The matrix $A$ differs from $A^{\prime}$ only in that it has extra zero rows and columns. Clearly, both matrices appear to be regular or irregular simultaneously.

Suppose that both matrices are regular. Taking into account that the vector $\bm{d}^{\prime}$ is obtained from $\bm{d}$ by removing zero components, we have
$$
\Delta^{\prime}
=
\sqrt{(A^{\prime}(\bm{d}^{\prime-}A^{\prime})^{-})^{-}\bm{d}^{\prime}}
=
\sqrt{(A(\bm{d}^{-}A)^{-})^{-}\bm{d}}
=
\Delta.
$$

Since the optimal vector $\bm{x}$ differs from $\bm{x}^{\prime}$ only in extra zero components, we conclude that $\rho(A\bm{x},\bm{d})$ achieves minimum at $\bm{x}=\Delta(\bm{d}^{-}A)^{-}$.
\end{proof}

In the next sections, we consider applications of the above result to analysis of linear dependence and to solution of linear equations.

\section{Linear Dependence}

First we give conditions for a vector $\bm{d}\in\mathbb{X}^{m}$ to be linearly dependent on vectors $\bm{a}_{1},\ldots,\bm{a}_{n}\in\mathbb{X}^{m}$, or equivalently, to admit a representation in the form of a linear combination $\bm{d}=x_{1}\bm{a}_{1}\oplus\cdots\oplus x_{n}\bm{a}_{n}$.

We define the matrix $A=(\bm{a}_{1},\ldots,\bm{a}_{n})$ and then calculate the residual $\Delta=\sqrt{(A(\bm{d}^{-}A)^{-})^{-}\bm{d}}$.

It follows from Lemma~\ref{L-Lb} that $\Delta\geq\mathbb{1}$. The equality $\Delta=\mathbb{1}$ means that the vector $\bm{d}$ belongs to the linear span $\mathcal{A}=\mathop\mathrm{span}\{\bm{a}_{1},\ldots,\bm{a}_{n}\}$, whereas the inequality $\Delta>\mathbb{1}$ implies that $\bm{d}$ is outside $\mathcal{A}$. In other words, we have the following statement.
\begin{lemma}\label{L-LH}
A vector $\bm{d}$ is linearly dependent on vectors $\bm{a}_{1},\ldots,\bm{a}_{n}$ if and only if $\Delta=\mathbb{1}$.
\end{lemma}

Now we formulate a criterion that a system $\bm{a}_{1},\ldots,\bm{a}_{n}$ is linearly independent. We denote by $A_{i}=(\bm{a}_{1},\ldots,\bm{a}_{i-1},\bm{a}_{i+1},\ldots,\bm{a}_{n})$ a matrix obtained from $A$ by removing column $i$, and introduce
$$
\delta(A)=\min_{1\leq i\leq n}\Delta_{A_{i}}(\bm{a}_{i}).
$$

\begin{lemma}\label{L-LIC}
The system of vectors $\bm{a}_{1},\ldots,\bm{a}_{n}$ is linearly independent if and only if $\delta(A)>\mathbb{1}$.
\end{lemma}
\begin{proof}
Clearly, the condition $\delta(A)>\mathbb{1}$ involves that $\Delta_{A_{i}}(\bm{a}_{i})>\mathbb{1}$ for all $i=1,\ldots,n$. It follows from Theorem~\ref{T-Lb} that here none of the vectors $\bm{a}_{1},\ldots,\bm{a}_{n}$ is a linear combination of others, and therefore, the system of the vectors is linearly independent.
\end{proof}

Let $\bm{a}_{1},\ldots,\bm{a}_{n}$ and $\bm{b}_{1},\ldots,\bm{b}_{k}$ be two systems of nonzero vectors. These systems are considered to be equivalent if each vector of one system is a linear combination of vectors of the other system.

Consider a system $\bm{a}_{1},\ldots,\bm{a}_{n}$ that can include linearly dependent vectors. To construct an equivalent linearly independent system, we implement a sequential procedure that examines the vectors one by one to decide whether to remove a vector or not.

At each step $i=1,\ldots,n$, the vector $\bm{a}_{i}$ is removed if $\Delta_{\widetilde{A}_{i}}(\bm{a}_{i})=\mathbb{1}$, where the matrix $\widetilde{A}_{i}$ is composed of those columns in $A_{i}$, that are retained after the previous steps. Upon completion of the procedure, we get a new system $\widetilde{\bm{a}}_{1},\ldots,\widetilde{\bm{a}}_{k}$, where $k\leq n$.

\begin{proposition}
The system $\widetilde{\bm{a}}_{1},\ldots,\widetilde{\bm{a}}_{k}$ is a linearly independent system that is equivalent to $\bm{a}_{1},\ldots,\bm{a}_{n}$.
\end{proposition}
\begin{proof}
According to the way of constructing the system $\widetilde{\bm{a}}_{1},\ldots,\widetilde{\bm{a}}_{k}$, each vector $\widetilde{\bm{a}}_{i}$ coincides with a vector of the system $\bm{a}_{1},\ldots,\bm{a}_{n}$. Since at the same time, for each $\bm{a}_{j}$, it holds that $\bm{a}_{j}\in\mathop\mathrm{span}\{\widetilde{\bm{a}}_{1},\ldots,\widetilde{\bm{a}}_{k}\}$, both systems are equivalent. Finally, due to Lemma~\ref{L-LIC}, the system $\widetilde{\bm{a}}_{1},\ldots,\widetilde{\bm{a}}_{k}$ is linearly independent.
\end{proof}

\section{Linear Equations}

Given a matrix $A\in\mathbb{X}^{m\times n}$ and a vector $\bm{d}\in\mathbb{X}^{m}$, consider the problem of finding an unknown vector $\bm{x}\in\mathbb{X}^{n}$ to satisfy the equation
\begin{equation}
A\bm{x}=\bm{d}.
\label{E-Axd}
\end{equation}

In what follows, we assume that the matrix $A$ is already put into a form where it is consistent with $\bm{d}$, and use the notation $\Delta=\Delta_{A}(\bm{d})$.

If a matrix $A=(\bm{a}_{1},\ldots,\bm{a}_{n})$ has a zero column, say $\bm{a}_{i}$, then the solution of equation \eqref{E-Axd} reduces to that of an equation that is obtained from \eqref{E-Axd} by removing the component $x_{i}$ in the vector $\bm{x}$ together with eliminating the column $\bm{a}_{i}$ in $A$. Each solution of the reduced equation causes equation \eqref{E-Axd} to have a set of solutions, where $x_{i}$ takes all values in $\mathbb{X}$. 

Suppose that $A=\mathbb{0}$. In this case, any vector $\bm{x}\in\mathbb{X}^{n}$ is a solution provided that $\bm{b}=\mathbb{0}$, and there is no solution otherwise. If $\bm{b}=\mathbb{0}$, then equation \eqref{E-Axd} has a trivial solution $\bm{x}=\mathbb{0}$, which is unique when the matrix $A$ has no zero columns.

From here on we assume that the vector $\bm{b}$ and all columns in the matrix $A$ are nonzero.


In the following, we examine conditions for the solution to exist and to be unique, and then describe the general solution to the equation. 

\subsection{Existence and Uniqueness of Solution}

Application of previous results brings us to a position to arrive at the next assertion. 
\begin{theorem}\label{T-EAxd}
If a matrix $A$ is consistent with a vector $\bm{d}\ne\mathbb{0}$, then the following statements are valid:
\begin{enumerate}[\rm(1)]
\item Equation \eqref{E-Axd} has solutions if and only if $\Delta=\mathbb{1}$.
\item If solvable, the equation has a solution
$$
\bm{x}
=
(\bm{d}^{-}A)^{-}.
$$
\item If all columns in $A$ form a minimal system that generates $\bm{b}$, then the above solution is unique.
\end{enumerate}
\end{theorem}
\begin{proof}
The existence condition and the form of a solution result from Theorem~\ref{T-Lb}. The uniqueness condition follows from representation of a vector as a unique linear combination of its minimal set of generators.
\end{proof}

We define a pseudo-solution to equation \eqref{E-Axd} as a vector that satisfies the equation $A\bm{x}=\Delta A(\bm{d}^{-}A)^{-}$.

Note that the last equation always has a solution that is given by $\bm{x}=\Delta(\bm{d}^{-}A)^{-}$. Clearly, when $\Delta=\mathbb{1}$ this pseudo-solution becomes an actual solution. Moreover, it follows from Theorem~\ref{T-Lb} that the pseudo-solution provides the minimum distance to the vector $\bm{d}$ over all vectors in the linear span of columns of the matrix $A$ in the sense of the metric $\rho$.

\subsection{General Solution}

To describe a general solution to equation \eqref{E-Axd}, we first prove an auxiliary result.

\begin{lemma}\label{L-AxdGS}
Suppose that $I$ is a subset of column indices of a matrix $A$ and $\bm{d}\in\mathop\mathrm{span}\{\bm{a}_{i}|i\in I\}$.

Then any vector $\bm{x}_{I}=(x_{i})$ with components $x_{i}=(\bm{d}^{-}\bm{a}_{i})^{-}$ if $i\in I$, and $x_{i}\leq(\bm{d}^{-}\bm{a}_{i})^{-}$ otherwise, is a solution to equation \eqref{E-Axd}.
\end{lemma}
\begin{proof}
To verify the statement, we first consider that $\bm{d}\in\mathop\mathrm{span}\{\bm{a}_{i}|i\in I\}\subset\mathop\mathrm{span}\{\bm{a}_{1},\ldots,\bm{a}_{n}\}$ and thus equation \eqref{E-Axd} has a solution $\bm{x}_{I}$. For the components $x_{i}$ of the solution, we can write
$$
\bm{d}
=
\bigoplus_{i=1}^{n}x_{i}\bm{a}_{i}
=
\bigoplus_{i\in I}x_{i}\bm{a}_{i}
\oplus
\bigoplus_{i\not\in I}x_{i}\bm{a}_{i}.
$$

We note that the condition $\bm{d}\in\mathop\mathrm{span}\{\bm{a}_{i}|i\in I\}$ is equivalent to an equality
$$
\bm{d}
=
\bigoplus_{i\in I}x_{i}\bm{a}_{i},
$$
which is valid when $x_{i}=(\bm{d}^{-}\bm{a}_{i})^{-}$ for all $i\in I$.

The remaining components with indices $i\not\in I$ must be taken so as to satisfy inequalities
$$
\bm{d}
\geq
\bigoplus_{i\not\in I}x_{i}\bm{a}_{i}
\geq
x_{i}\bm{a}_{i}.
$$

It remains to solve the inequalities to conclude that for all $i\not\in I$, we can take any $x_{i}\leq(\bm{d}^{-}\bm{a}_{i})^{-}$.
\end{proof}


Now suppose that the columns with indices in $I$ form a minimal generating system for $\bm{d}$. Denote by $\mathcal{I}$ a set of all such index sets $I$. Clearly, $\mathcal{I}\neq\emptyset$ only when equation \eqref{E-Axd} has at least one solution.

By applying Lemma~\ref{L-AxdGS}, it is not difficult to verify that the following statement holds.

\begin{theorem}\label{T-GS}
The general solution to equation \eqref{E-Axd} is a (possible empty) family of solutions $\bm{x}=\{\bm{x}_{I}| I\in\mathcal{I}\}$, where each solution $\bm{x}_{I}=(x_{i})$ is given by
\begin{align*}
x_{i}
&=
(\bm{d}^{-}\bm{a}_{i})^{-},
\quad
\text{if $i\in I$},
\\
x_{i}
&\leq
(\bm{d}^{-}\bm{a}_{i})^{-},
\quad
\text{if $i\not\in I$}.
\end{align*}
\end{theorem}

Consider a case when the family reduces to one solution set. Let the columns in $A$ are linearly independent. Then there may exist only one subset of columns that form a minimal generating system for $\bm{d}$. If the subset coincides with the set of all columns, then the solution reduces to a unique vector $\bm{x}=(\bm{d}^{-}A)^{-}$.

Graphical illustration of unique and non-unique solutions to equation \eqref{E-Axd} are given in Fig.~\ref{F-AxdS}.
\begin{figure}[ht]
\setlength{\unitlength}{1mm}
\begin{center}

\begin{picture}(25,47)

\put(0,5){\vector(1,0){25}}
\put(5,2){\vector(0,1){45}}

\put(5,5){\thicklines\vector(1,4){3}}
\put(9,17){\line(-1,0){5}}

\put(2,11){\thicklines\line(1,1){21}}
\multiput(3,12)(1,1){20}{\line(1,0){1}}

\put(5,5){\thicklines\vector(3,1){12}}
\put(17,9){\line(0,-1){5}}

\put(10,2){\thicklines\line(1,1){16}}
\multiput(11,3)(1,1){16}{\line(-1,0){1}}

\put(5,5){\thicklines\vector(3,4){18}}
\put(23,29){\line(-1,0){19}}
\put(23,29){\line(0,-1){25}}

\put(13,12){$\bm{a}_{1}$}
\put(6,21){$\bm{a}_{2}$}
\put(24,30){$\bm{d}$}

\put(18,1){$x_{1}$}
\put(0,23){$x_{2}$}

\end{picture}
\hspace{20\unitlength}
\begin{picture}(25,47)

\put(0,5){\vector(1,0){25}}
\put(5,2){\vector(0,1){45}}

\put(5,5){\thicklines\vector(1,4){3}}
\put(9,17){\line(-1,0){5}}

\put(2,11){\thicklines\line(1,1){24}}
\multiput(3,12)(1,1){23}{\line(1,0){1}}

\put(5,5){\thicklines\vector(3,1){12}}
\put(17,9){\line(0,-1){5}}

\put(10,2){\thicklines\line(1,1){16}}
\multiput(11,3)(1,1){16}{\line(-1,0){1}}

\put(5,5){\thicklines\vector(2,3){18}}
\put(23,32){\line(-1,0){19}}
\put(23,32){\line(0,-1){28}}

\put(13,12){$\bm{a}_{1}$}
\put(6,21){$\bm{a}_{2}$}
\put(20,34){$\bm{d}$}

\put(18,1){$x_{1}$}
\put(0,24){$x_{2}$}

\end{picture}
\hspace{20\unitlength}
\begin{picture}(25,47)

\put(0,5){\vector(1,0){25}}
\put(5,2){\vector(0,1){45}}

\put(5,5){\thicklines\vector(0,1){16}}
\put(13,29){\line(-1,0){9}}

\put(2,18){\thicklines\line(1,1){22}}
\put(2.1,18){\thicklines\line(1,1){22}}
\put(1.9,18){\thicklines\line(1,1){22}}

\put(5,5){\thicklines\vector(1,3){8}}

\put(5,5){\thicklines\vector(1,2){16}}
\put(21,37){\line(-1,0){17}}
\put(21,37){\line(0,-1){33}}

\put(0,23){$\bm{a}_{1}$}
\put(8,31){$\bm{a}_{2}$}
\put(18,39){$\bm{d}$}

\put(12,1){$x_{1}$}
\put(0,32){$x_{2}$}

\end{picture}
\end{center}
\caption{A unique (left) and non-unique (middle and right) solutions to linear equations in $\mathbb{R}_{\max,+}^{2}$.}\label{F-AxdS}
\end{figure}
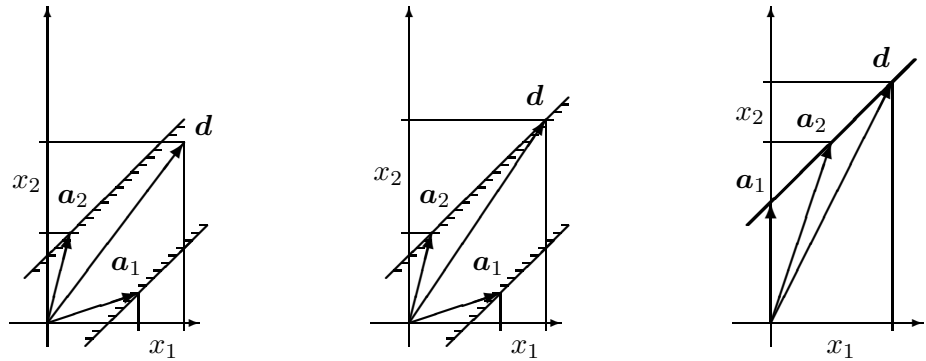

\bibliographystyle{utphys}

\bibliography{A_solution_of_a_tropical_linear_vector_equation}

\end{document}